\documentclass[12pt]{article}

\usepackage{amsmath, amsthm, amssymb}

\usepackage{fullpage}

\usepackage{hyperref}

\newtheorem{theorem}{Theorem}
\newtheorem{lemma}{Lemma}

\theoremstyle{definition}

\theoremstyle{remark}

\newcommand{\cF}{\mathcal{F}}

\title{A remark on the $t$-intersecting Erd\H{o}s-Ko-Rado theorem}
\author{William Linz\thanks{University of South Carolina, Columbia, SC, USA. ({\tt wlinz@mailbox.sc.edu}). Partially supported by NSF RTG Grant DMS 2038080.}}
\date{\today}

\begin{document}
\maketitle

\begin{abstract}
The $t$-intersecting Erd\H{o}s-Ko-Rado theorem is the following statement: if $\mathcal{F} \subset \binom{[n]}{k}$ is a $t$-intersecting family of sets and $n\ge (t+1)(k-t+1)$, then $|\cF| \le \binom{n-t}{k-t}$. The first proof of this statement for all $t$ was a linear algebraic argument of Wilson. Earlier, Schrijver had proven the $t$-intersecting Erd\H{o}s-Ko-Rado theorem for sufficiently large $n$ by a seemingly different linear algebraic argument motivated by Delsarte theory. In this note, we show that the approaches of Schrijver and Wilson are in fact equivalent.
\end{abstract}

For integers $n > k > 0$, $\binom{[n]}{k}$ denotes the family of $k$-element subsets of $[n]:=\{1, \ldots, n\}$. A family of sets $\cF\subset \binom{[n]}{k}$ is \emph{$t$-intersecting} if $|F\cap F'| \ge t$ for all $F, F'\in \mathcal{F}$. The maximum size of a $t$-intersecting family $\mathcal{F} \subset \binom{[n]}{k}$ is given by the $t$-intersecting Erd\H{o}s-Ko-Rado theorem, which was proved by Erd\H{o}s, Ko and Rado~\cite{EKR} for sufficiently large $n$, in the optimal range $n\ge (t+1)(k-t+1)$ for $t\ge 15$ by Frankl~\cite{Frankl}, and finally by Wilson~\cite{Wil} for all $t$. 

\begin{theorem}[$t$-intersecting Erd\H{o}s-Ko-Rado]\label{thm:tintekrset}
Let $n, k, t$ be integers. Suppose $\cF\subset \binom{[n]}{k}$ is a $t$-intersecting family with $n\ge (t+1)(k-t+1)$. Then, 
\[|\cF| \le \binom{n-t}{k-t}.\]
If $n > (t+1)(k-t+1)$, equality holds only if $\cF\cong\{F\in \binom{[n]}{k}: [t] \subset F\}$. 
\end{theorem}

Let $G(n, k, t)$ be the graph with vertex set $\binom{[n]}{k}$ and edge set $\{\{F, F'\}: |F\cap F'| < t\}$. By construction, $t$-intersecting families are independent sets in $G(n, k, t)$, so the $t$-intersecting Erd\H{o}s-Ko-Rado theorem can be restated as 
\[\alpha(G(n, k, t)) \le \binom{n-t}{k-t}.\]
Wilson's proof actually gives the stronger result that 
\[\vartheta(G(n, k, t)) = \binom{n-t}{k-t},\]
where $\vartheta(G)$ is the Lov\'asz number~\cite{Lov} of the graph $G$. It follows from the sandwich theorem of Lov\'asz that the Shannon capacity of $G(n, k, t)$ is also $\binom{n-t}{k-t}$ when $n\ge (t+1)(k-t+1)$. 

We briefly outline Wilson's proof. For a graph $G$ on $n$ vertices, a \emph{pseudoadjacency matrix} is an $n\times n$ matrix $B=(b_{ij})$ with constant row sums and such that $b_{ij} = 0$ when $i\not\sim j$. The Hoffman bound holds for any pseudoadjacency matrix of a graph; that is, if $B$ is a pseudoadjacency matrix of a graph $G$, and $B$ has eigenvalues $\lambda_1 \ge \ldots \ge \lambda_n$, then 
\[\alpha(G) \le \frac{-\lambda_n}{\lambda_1 - \lambda_n} n.\]
Wilson constructed the matrix \[\Omega(n, k, t) = \sum_{i=0}^{t-1}(-1)^{t-1-i}\binom{k-1-i}{k-t}\binom{n-k-t+i}{k-t}^{-1}D_{k-i},\]
where if $\alpha$ and $\beta$ are two $k$-subsets, then
\[(D_i)_{\alpha, \beta} = \binom{|\alpha \setminus \beta|}{i} = \binom{k-|\alpha \cap \beta|}{i}.\] It is easy to see that $\Omega(n, k, t)$ is a pseudoadjacency matrix for $G(n, k, t)$. Furthermore, Wilson showed that $\lambda_1(\Omega(n, k, t)) = \binom{n}{k}\binom{n-t}{k-t}^{-1} - 1$ and $\lambda_n(\Omega(n, k, t))= -1$, so the Erd\H{o}s-Ko-Rado theorem follows by an application of the Hoffman bound. 

The definition of the matrix $\Omega(n, k, t)$ is rather unmotivated in Wilson's proof. Godsil and Guo~\cite{GG} gave a different interpretation of how the matrix $\Omega(n, k, t)$ may be constructed based on a matrix defined in the book of Godsil and Meagher \cite[pg.~159]{GM}; the proof that the construction of Godsil and Meagher is equivalent to $\Omega(n, k, t)$ relies on Keevash's theorem~\cite{Kee} on the existence of $t$-$(n, k, 1)$ designs. The purpose of this note is to give an alternate construction of the matrix $\Omega(n, k, t)$ using an earlier formulation of Schrijver~\cite{Sch1978} that is itself motivated by results in Delsarte theory~\cite{Del} and to prove that the matrix constructed is the same as Wilson's without using Keevash's theorem. 

Schrijver~\cite{Sch1978} had in fact proved some years before Wilson that $\vartheta(G(n, k, t)) = \binom{n-t}{k-t}$ as $n \rightarrow \infty$. Schrijver's argument uses a linear programming formulation of the Lov\'asz number for association schemes \cite{Sch}, but the argument can be reframed as constructing a pseudoadjacency matrix of $G(n, k, t)$.  Let $A_i$ be the $\binom{[n]}{k} \times \binom{[n]}{k}$ matrix with rows and columns indexed by the sets in $\binom{[n]}{k}$, with $(A_i)_{(F, F')} =  1$  if $ |F\cap F'| = k -i$ and $(A_i)_{(F, F')} = 0$ otherwise.

Following Schrijver, we define the pseudoadjacency matrix 
\[S(n, k, t) = \sum_{i=0}^{t-1}\frac{a_{k-i}}{\binom{k}{k-i}\binom{n-k}{k-i}}A_{k-i},\]
where
\begin{equation}\label{akellieqn}
a_{k-i} = \frac{1}{\binom{n-t}{k-t}}\binom{k}{k-i}\sum_{j=0}^{k-i}(-1)^{k-i-j}\binom{k-i}{j}\binom{n-\min(k-j, t)}{n-k}.
\end{equation}

Our main theorem is that the matrix $S(n, k, t)$ is equivalent to Wilson's matrix. 

\begin{theorem}\label{thm:somegathm}
For integers $n > k > t > 0$, we have 
\[S(n, k, t) = \Omega(n, k, t).\]
\end{theorem}

Before we give the proof of Theorem~\ref{thm:somegathm}, we motivate the construction of the matrix $S(n, k, t)$ and remark on how it fits with many classic results about association schemes and $t$-intersecting Erd\H{o}s-Ko-Rado theorems. 

The matrices $\{A_i\}_{0\le i\le k}$ form a basis for the Bose-Mesner algebra of the Johnson scheme; in particular, $(A_i)_{F, F'} = 1$ if and only if the two sets $F, F'\subset \binom{[n]}{k}$ are at distance $i$ in the Johnson graph $J(n, k)$. For a family $Y \subset \binom{[n]}{k}$, the \emph{inner distribution} with respect to the Johnson scheme is the vector ${\bf e} = (e_0, e_1, \ldots, e_k)$, with entries given by \[e_i = \frac{\chi_Y^TA_i\chi_Y}{|Y|},\]
where $\chi_Y$ is the characteristic vector of the family $Y$. The numbers $a_{k-i}$ in \eqref{akellieqn} are the inner distribution of a $t$-$(n, k, 1)$ design when such a design exists, though of course the $a_{k-i}$ are defined for all $n$. It follows from Delsarte theory~\cite[Theorem 3.3]{Del} that if there exists a $t$-$(n, k, 1)$ design, then the matrix $S(n, k, t) + I$ is positive semidefinite. The challenge of proving the $t$-intersecting Erd\H{o}s-Ko-Rado theorem is to show that $S(n, k, t) + I$ is positive semidefinite for all $n\ge (t+1)(k-t+1)$. 

Erd\H{o}s-Ko-Rado theorems for $t$-intersecting families have been proven for many of the other infinite families of $Q$-polynomial association schemes, including the Hamming scheme~\cite{Moo} and the Grassmann scheme~\cite{FW}. Tanaka~\cite{Tan1, Tan2} has given a unified approach to deriving many of these Erd\H{o}s-Ko-Rado theorems though the lens of solving certain linear programming problems arising from Delsarte theory. In fact, Tanaka~\cite{Tan3} has constructed the optimal feasible solution to these linear programming problems and shown their uniqueness as applications of much more general results from the theory of Leonard systems.  The content of Theorem~\ref{thm:somegathm} is that the solution of the corresponding linear programming problem for the Johnson scheme comes from the inner distribution of a $t$-$(n, k, 1)$ design and that the matrix $\Omega(n, k, t)$ is simply the same matrix expressed in a different basis of the Bose-Mesner algebra of the Johnson scheme. 

Returning to the matrix $\Omega(n, k, t)$, it is well-known that the matrices $\{D_j\}_{0\le j\le k}$ form another basis for the Bose-Mesner algebra of the Johnson scheme. The following lemma relates the matrices $\{A_i\}_{0\le i\le k}$ and $\{D_j\}_{0\le j\le k}$ \cite[pg.121]{GM}.

\begin{lemma}\label{lem:aidj}
\[A_i = \sum_{r=i}^{k}(-1)^{r-i}\binom{r}{i}D_r.\]
\end{lemma}

We conclude by giving an elementary proof of Theorem~\ref{thm:somegathm} that only uses binomial coefficient manipulations. 

\begin{proof}[Proof of Theorem~\ref{thm:somegathm}]

It suffices to prove that for every $0\le i\le t-1$, 
\begin{equation}\label{t2eqn}
\sum_{j=i}^{t-1} \frac{a_{k-j}}{\binom{k}{k-j}\binom{n-k}{k-j}} (-1)^{j-i} \binom{k-i}{k-j} = (-1)^{t-1-i} \binom{k-1-i}{k-t}\binom{n-k-t+i}{k-t}^{-1}.
\end{equation}

Indeed, by Lemma~\ref{lem:aidj} we have that 
\[A_{k-i} = \sum_{r=k-i}^k(-1)^{r-k+i}\binom{r}{k-i} D_r,\]

so 
\begin{align*}
S(n, k, t) &= \sum_{i=0}^{t-1}\frac{a_{k-i}}{\binom{k}{k-i}\binom{n-k}{k-i}}A_{k-i}\\
& = \sum_{i=0}^{t-1}\frac{a_{k-i}}{\binom{k}{k-i}\binom{n-k}{k-i}}\sum_{r=k-i}^k(-1)^{r-k+i}\binom{r}{k-i} D_r\\
&= \sum_{i=0}^{t-1} D_{k-i} \sum_{j=i}^{t-1} \frac{a_{k-j}}{\binom{k}{k-j}\binom{n-k}{k-j}} (-1)^{j-i} \binom{k-i}{k-j}.
\end{align*}

We have by \eqref{akellieqn}
\begin{align*}
a_{k-j} &= \binom{k}{k-j}(-1)^{j+t}\binom{k-j-1}{k-t} + \frac{\binom{k}{k-j}}{\binom{n-t}{k-t}} \sum_{r=k-t+1}^{k-j}(-1)^{k-j-r}\binom{k-j}{r}\binom{n-k+r}{r}\\
&=  \binom{k}{k-j}(-1)^{j+t}\binom{k-j-1}{k-t} + \frac{\binom{k}{k-j}}{\binom{n-t}{k-t}}\sum_{r=1}^{t-j}(-1)^{t-r}\binom{k-j}{t-j-r}\binom{n-t+r}{k-t+r}.
\end{align*}

The left hand side of \eqref{t2eqn} is  
\[
S_i =\frac{1}{\binom{n-t}{k-t}} \sum_{j=i}^{t-1}(-1)^i\binom{k-i}{j-i}\left(\frac{(-1)^{t}\binom{n-t}{k-t}\binom{k-j-1}{k-t} + \sum_{r=1}^{t-j}(-1)^{t-r}\binom{k-j}{t-j-r}\binom{n-t+r}{k-t+r}}{\binom{n-k}{k-j}}\right)\]
Using the fact that $\binom{n-k}{k-i}\binom{k-i}{j-i}   = \binom{n-2k+j}{j-i}\binom{n-k}{k-j}$, we obtain
\begin{equation}\label{si}
S_i = \frac{(-1)^i}{\binom{n-t}{k-t}\binom{n-k}{k-i}}M_i
= \frac{(-1)^i}{\binom{n-t}{k-i}\binom{n-k-t+i}{k-t}}M_i
\end{equation}
where
\begin{equation}\label{mi}
M_i = \sum_{j=i}^{t-1}\binom{n-2k+j}{j-i}\left((-1)^{t}\binom{n-t}{k-t}\binom{k-j-1}{k-t} + \sum_{r=1}^{t-j}(-1)^{t-r}\binom{k-j}{t-j-r}\binom{n-t+r}{k-t+r}\right)
\end{equation}

By Vandermonde's identity,
\[(-1)^{t}\binom{n-t}{k-t}\sum_{j=i}^{t-1}\binom{n-2k+j}{j-i}\binom{k-j-1}{t-j-1} = (-1)^t\binom{n-t}{k-t}\binom{n-k-1}{t-i-1}.\]

Now consider the double sum
\[\sum_{j=i}^{t-1}\binom{n-2k+j}{j-i}\sum_{r=1}^{t-j}(-1)^{t-r}\binom{k-j}{t-j-r}\binom{n-t+r}{k-t+r}.\]

Interchanging the order of summation and using Vandermonde's identity again, this is 
\begin{align*}
&\sum_{r=1}^{t-i}(-1)^{t-1}\binom{n-t+r}{k-t+r} \sum_{j=i}^{t-r}\binom{k-j}{t-j-r}\binom{n-2k+j}{j-i}\\
&=\sum_{r=1}^{t-i}(-1)^{t-r}\binom{n-t+r}{k-t+r}\binom{n-k}{t-r-i}.
\end{align*}

So we have
\begin{equation}\label{mainsum}
M_i = (-1)^t\binom{n-t}{k-t}\binom{n-k-1}{t-i-1} + \sum_{r=1}^{t-i}(-1)^{t-r}\binom{n-t+r}{k-t+r}\binom{n-k}{t-r-i}.
\end{equation}

We now look at the term 
\begin{equation}\label{eqn:n-t+r}
 \sum_{r=1}^{t-i}(-1)^{t-r}\binom{n-t+r}{k-t+r}\binom{n-k}{t-r-i}.
 \end{equation}

By using Pascal's identity repeatedly (or by using Vandermonde's identity), we have
\begin{equation}\label{eqn:n-t}
\binom{n-t+r}{k-t+r} = \sum_{s=0}^{r}\binom{r}{s}\binom{n-t}{k-t+r-s}.
\end{equation}
Substituting \eqref{eqn:n-t} into \eqref{eqn:n-t+r}, the sum becomes 
\begin{equation}\label{eqn:n-tsub}
\sum_{r=1}^{t-i}(-1)^{t-r}\binom{n-k}{t-r-i}\sum_{s=0}^{r}\binom{r}{s}\binom{n-t}{k-t+r-s}.
\end{equation}
We now rearrange the double sum of \eqref{eqn:n-tsub} so that the sum is of the form $\sum_{m=0}\binom{n-t}{k-t+m}a_m$. The sum becomes 
\begin{equation}\label{eqn:n-tk-t+m}
\binom{n-t}{k-t}\sum_{r=1}^{t-i}(-1)^{t-r}\binom{n-k}{t-r-i} + \sum_{m=1}^{t-i}\binom{n-t}{k-t+m} \sum_{r=m}^{t-i}(-1)^{t-r}\binom{n-k}{t-r-i}\binom{r}{m}.
\end{equation}
We look at the first part of \eqref{eqn:n-tk-t+m}. We have
\begin{align*}
\binom{n-t}{k-t}\sum_{r=1}^{t-i}(-1)^{t-r}\binom{n-k}{t-r-i} 
&= \binom{n-t}{k-t}\sum_{r=0}^{t-i-1}(-1)^{r+i}\binom{n-k}{r}\\
&=\binom{n-t}{k-t}(-1)^{i}\sum_{r=0}^{t-i-1}(-1)^r\binom{n-k}{r}\\
&=(-1)^{i}\binom{n-t}{k-t}(-1)^{t-i-1}\binom{n-k-1}{t-i-1},
\end{align*}
which cancels out the first term of \eqref{mainsum}. Thus, the sum $M_i$ may be simplified as
\begin{equation}\label{sisimplified}
M_i = \sum_{m=1}^{t-i}\binom{n-t}{k-t+m} \sum_{r=m}^{t-i}(-1)^{t-r}\binom{n-k}{t-r-i}\binom{r}{m}.
\end{equation}
The inner sum of \eqref{sisimplified} may be simplified as 
\begin{align*}
\sum_{r=m}^{t-i}(-1)^{t-r}\binom{n-k}{t-r-i}\binom{r}{m} &= \sum_{r=0}^{t-m-i}(-1)^{r+i}\binom{n-k}{r}\binom{t-i-r}{m}\\
&= (-1)^i\sum_{r=0}^{t-m-i}\binom{k-n+r-1}{r}\binom{t-i-r}{t-i-r-m}\\
&=(-1)^i\binom{k-n+t-i-1}{t-i-m}\\ &= (-1)^i \frac{(k-n+t-i-1)\cdots (k-n+m)}{(t-i-m)!}.\\
\end{align*}
Now observe that 
\begin{equation}\label{eqn:n-tk-i}
\binom{n-t}{k-t+m} = \binom{n-t}{k-i} \frac{(k-t+m+1)\cdots(k-t+(t-i))}{(n-k-m)\cdots(n-k-(t-i)+1)}
\end{equation}

Substituting \eqref{eqn:n-tk-i} and the previous result on the inner sum into \eqref{sisimplified}, we obtain that
\begin{align*}
M_i &= (-1)^i\binom{n-t}{k-i}\sum_{m=1}^{t-i}\frac{(k-t+m+1)\cdots(k-t+(t-i))}{(n-k-m)\cdots(n-k-(t-i)+1)}\binom{k-n+t-i-1}{t-i-m}\\
&=(-1)^i\binom{n-t}{k-i}\sum_{m=1}^{t-i}(-1)^{t-i-m}\binom{k-i}{t-i-m}\\
&=(-1)^i\binom{n-t}{k-i}\sum_{m=0}^{t-i-1}(-1)^m\binom{k-i}{m}\\
&=(-1)^i\binom{n-t}{k-i}(-1)^{t-i-1}\binom{k-i-1}{t-i-1}\\
&=\binom{n-t}{k-i}(-1)^{t-1}\binom{k-i-1}{k-t}.
\end{align*}
It follows upon substitution into \eqref{si} that 
\[S_i = \frac{(-1)^i}{\binom{n-t}{k-i}\binom{n-k-t+i}{k-t}}\binom{n-t}{k-i}(-1)^{t-1}\binom{k-i-1}{k-t} = (-1)^{t-1-i} \binom{k-1-i}{k-t}\binom{n-k-t+i}{k-t}^{-1},\]
which proves \eqref{t2eqn} and completes the proof of Theorem~\ref{thm:somegathm}. 
\end{proof}

\section*{Acknowledgement}
I thank Hajime Tanaka for bringing the reference \cite{Tan3} to my attention.

\end{document}